\documentclass[preprint,12pt,1p]{elsarticle}

\makeatother
\usepackage{graphics}
\usepackage{epsfig}
\usepackage{mathptmx}
\usepackage{amsmath}
\usepackage{amssymb}
\usepackage{hyperref}
\usepackage{fullpage}
\usepackage{amsthm}
\usepackage{lineno}
\usepackage{geometry}
\theoremstyle{plain}
\newtheorem{theorem}{Theorem}[section]

\newtheorem{lemma}[theorem]{Lemma}

\theoremstyle{definition}
\newtheorem{definition}{Definition}[section]
\theoremstyle{example}

\theoremstyle{remark}

\newlength{\defbaselineskip}
\makeatletter
\def\ps@pprintTitle{
  \let\@oddhead\@empty
  \let\@evenhead\@empty
  \let\@oddfoot\@empty
  \let\@evenfoot\@oddfoot
}
\journal{Computer aided geometric design}
\begin{document}
\begin{frontmatter}
\title{HEMI-SLANT SUBMANIFOLDS OF COSYMPLECTIC MANIFOLDS}
\author[n1]{Mehraj Ahmad Lone\corref{cor1}}
\ead{mehraj.jmi@gmail.com}
\author[n2]{Mohamd Saleem Lone}
\author[n3]{Mohammad Hasan Shahid}
\address[n1]{Department of Mathematics, Central University of Jammu, Jammu, 180011, India.}
\address[n2]{Department of Mathematics, Central University of Jammu, Jammu, 180011, India.}
\address[n3]{Department of Mathematics, Jamia Millia Islamia, New Delhi-110 025, India}
\cortext[cor1]{Corresponding author}
\begin{abstract}
In this paper we study the hemi-slant submanifolds of cosymplectic manifolds. Necessary and sufficient conditions for distributions to be integrable are worked out. Some important results are obtained in this direction.
\end{abstract}
\begin{keyword}
\texttt Slant submanifolds, Cosymplectic manifolds, Hemi-slant submanifolds

\end{keyword}
\end{frontmatter}
\section{Introduction}
In 1990, B. Y. Chen introduced the notion of  slant submanifold, which generalizes holomorphic and totally real submanifolds \cite{biha1}.  After that many research articles have been published by different geometers in this direction for different ambient spaces.
\newline

A. Lotta introduced  the notion of slant immersions of a Riemannian manifolds into an almost contact metric manifolds \cite{biha3}. After, these submanifolds were studied by J. L.Cabrerizo \emph{et. al} in the setting of Sasakian manifolds \cite{biha7}. In \cite{biha8} Papaghiuc  defines the semi-slant submanifolds as a generilization of slant submanifolds . Bislant submanifolds of an almost Hermitian manifold were introduced as natural generalization of semi-slant submanifolds by Carriazo \cite{biha2}. One of the classes of bi-slant submanifolds is that of anti-slant submanifolds which are studied by A. Carriazo \cite{biha2} but the name anti-slant seems to refer that it has no slant factor, so B. Sahin \cite{biha4} give the name of hemi-slant submanifolds  instead of anti-slant submanifolds. In \cite{biha2} V. A. Khan and  M. A. Khan studied the hemi-slant submanifolds of sasakian manifolds.
\newline

In this paper we study the hemi-slant submanifolds of cosymplectic manifolds. In section $2$, we collect the basic formulae and definitions for a cosymplectic manifolds and their submanifolds for ready references. In section $3$, we study the  hemi-slant submanifolds of cosymplectic manifolds. We obtain the integrability conditions  of the distributions  which are involved in the definition.

\section{ Preliminaries }
Let ${N}$ be a $(2m+1)$-dimensional almost contact metric manifold with structure $(\phi,\xi, \eta, g)$ where $\phi$ is a tensor field of type $(1,1)$, $\xi$  a vector field, $\eta$ is a one form and $g$ is the Riemannian metric on $N$. Then they satisfy
\begin{eqnarray}\label{a1}
\phi^{2} = -I + \eta \otimes \xi,\hspace{.5cm} \eta(\xi) = 1, \hspace{.5cm}g(\phi X, \phi Y) = g(X,Y) - \eta(X)\eta(Y).
\end{eqnarray}
These conditions also imply that
\begin{eqnarray}\label{a2}
\phi{\xi}=0,\hspace{.5cm}\eta(\phi X) = 0, \hspace{.5cm}\eta(X) = g(X,\xi),
\end{eqnarray}
and
\begin{eqnarray}\label{a3}
g(\phi X, Y) + g(X, \phi Y) = 0,
\end{eqnarray}
for all vector fields $X, Y$ in $TN$. Where $TN$ denotes the Lie algebra of vector fields on $N$. A normal almost contact metric manifold is called a cosymplectic manifold if
\begin{eqnarray}\label{a4}
(\overline{\nabla}_{X}\phi) = 0,\hspace{.5cm} \overline{\nabla}_{X}\xi = 0,
\end{eqnarray}
where $\overline{\nabla}$ denotes the Levi-Civita connection of $(N,g)$.

Throughout, we denote by $N$ a cosymplectic manifold, $M$ a submanifold of $N$ and $\xi$ a  structure vector field tangent to $M$. $A$ and $h$ denotes the shape operator and second fundamental form of immersion of $M$ into $N$. If $\nabla$ is the induced connection on $M$, the Gauss and Weingarten formulae of $M$ into $N$ are then given respectively by
\begin{eqnarray}\label{a5}
\overline{\nabla}_{X}Y = \nabla_{X}Y + h(X,Y),
\end{eqnarray}
\begin{eqnarray}\label{a6}
\overline{\nabla}_{X}V = -A_{V}X + \nabla_{X}^{\perp}V,
\end{eqnarray}
for all vector fields $X, Y$ on $TM$ and $V$ on $T^{\perp}M$, where $\nabla^{\perp}$ denotes the connection on the normal bundle $T^{\perp}M$ of $M$. The shape operator and the second fundamental form are related by
\begin{eqnarray}\label{a7}
g(A_{V}X,Y) = g(h(X,Y),V).
\end{eqnarray}
The mean curvature vector is defined by
\begin{eqnarray}\label{a8}
H = \frac{1}{n} trace (h) = \frac{1}{n}\sum_{i=1}^{n} h(e_{i},e_{i}),
\end{eqnarray}
where $n$ is the dimension of $M$ and $\{e_{1},e_{2}, ... ,  e_{n}\}$ is the local orthonormal frame of $M$.

For any $X\in TM$, we can write
\begin{eqnarray}\label{a9}
\phi X = TX + FX,
\end{eqnarray}
where $TX$ and $FX$ are the tangential and normal components of $\phi X$ respectively.

Similarly for any $V\in T^{\perp}M$, we have
\begin{eqnarray}\label{a10}
\phi V = tV + fV,
\end{eqnarray}
where $tV$ and $fV$ are the tangential and normal components of $\phi V$ respectively.
%
%
The covariant derivative of the tensor fields $T$, $F$, $t$ and $f$ are defined by the following
\begin{eqnarray}\label{a13}
(\nabla_{X}T)Y = \nabla_{X}TY - T\nabla_{X}Y,
\end{eqnarray}
\begin{eqnarray}\label{a14}
(\nabla_{X}F)Y = \nabla_{X}^{\perp}FY - F\nabla_{X}Y,
\end{eqnarray}
\begin{eqnarray}\label{a15}
(\nabla_{X}t)V = \nabla_{X}tV - t\nabla_{X}^{\perp}V,
\end{eqnarray}
and
\begin{eqnarray}\label{a16}
(\nabla_{X}f)V = \nabla_{X}^{\perp}fV - f\nabla_{X}^{\perp}V,
\end{eqnarray}
for all $X$, $Y$ $\in$ $TM$ and $V \in T^{\perp}M$.

A submanifold $M$ of an almost contact metric manifold $N$ is said to be totally umbilical if
\begin{eqnarray}\label{a17}
h(X,Y) = g(X,Y)H,
\end{eqnarray}
where $H$ is the mean curvature vector. If $h(X,Y) = 0$ for any $X,Y \in TM$, then $M$ is said to be totally geodesic and  if $H = 0$, then $M$ is said to be a minimal submanifold.

A. Lotta has introduced the notion of slant immersion of a Riemannian manifold into an almost contact metric manifold \cite{biha3} and slant submanifolds in Sasakian manifolds have been studied by J.L. Cabrerizo et al. \cite{biha7}.

For any $x \in M$ and $X\in T_{x}M$ if the vectors $X$ and $\xi$ are linearly independent, the angle denoted by $\theta (X) \in [0,\frac{\pi}{2}]$ between $\phi X$ and $T_{x}M$ is well defined. If $\theta(X)$ does not depend on the choice of $x \in M$ and $X\in T_{x}M$, we say that $M$ is slant in $N$ .
The constant angle $\theta$ is then called the slant angle of $M$ in $N$ . The anti-invariant submanifold of an almost contact metric manifold is a slant submanifold with slant angle $\theta =\frac{\pi}{2}$ and an invariant submanifold is a slant submanifold with the slant angle $\theta = 0 $. If the slant angle $\theta$ of $M$ is different from 0 and $\frac{\pi}{2}$ , then it is called a proper slant submanifold. If $M$ is a slant submanifold of an almost contact manifold then the tangent bundle $TM$ of $M$ is decomposed as
\begin{eqnarray*}\label{a17}
TM = D \oplus \langle\xi\rangle,
\end{eqnarray*}
where $\langle\xi\rangle$ denotes the distribution spanned by the structure vector field $\xi$ and $D$ is a complementary distribution of $\langle\xi\rangle$ in $TM$, known as the slant distribution.
For a proper slant submanifold $M$ of an almost contact manifold $N$ with a slant angle $\theta$, Lotta \cite{biha3} proved that
\begin{eqnarray*}\label{a18}
T^{2}X =  -cos^{2}\theta (X  - \eta(X)\xi),\hspace{.7cm} \forall X \in TM.
\end{eqnarray*}
Cabrerizo et al. \cite{biha7} extended the above result into a characterization for a slant submanifold in a contact metric manifold. In fact, they obtained the following crucial theorems.
\begin{theorem}\cite{biha7}
Let $M$ be a slant submanifold of an almost contact metric manifold $N$ such that $\xi \in TM$. Then $M$ is slant submanifold if and only if there exist a constant $\lambda \in [0,1]$ such that
\begin{eqnarray*}\label{a19}
T^{2} = -\lambda(I - \eta \otimes\xi),
\end{eqnarray*}
furthermore, in such case, if $\theta$ is the slant angle of $M$, then $\lambda = cos^{2}\theta$.
\end{theorem}
\begin{theorem}\cite{biha7}
Let $M$ be a slant submanifold of an almost contact metric manifold $\overline {M}$ with slant angle $\theta$. Then for any $X,Y \in TM$, we have
\begin{eqnarray*}\label{a20}
g(TX,TY) = cos^{2}\theta\{g(X,Y)  - \eta(X)\eta(Y)\},
\end{eqnarray*}
and
\begin{eqnarray*}\label{a21}
g(FX,FY) = sin^{2}\theta\{g(X,Y)  - \eta(X)\eta(Y)\}.
\end{eqnarray*}
\end{theorem}

\section{Hemi-slant submanifolds of cosymplectic manifolds}
In the present section, we introduce the hemi-slant submanifolds and obtain the necessary and sufficient conditions for the distributions of hemi-slant submanifolds of cosymplectic manifolds to be integrable.
\begin{definition}
Let $M$ be submanifold of an almost contact metric manifold $N$, then $M$ is said to be a hemi-slant submanifold if there exist two orthogonal distributions $D^{\theta}$ and $D^{\perp}$ on $M$ such that

(i) $TM$ = $D^{\theta}\oplus D^{\perp}$ $\oplus$  $\langle\xi\rangle $

(ii) $D^{\theta}$ is a slant distribution with slant angle $\theta \ne \frac{\pi}{2}$,

(iii) $D^{\perp}$ is a totally real,that is $JD^{\perp}\subseteq T^{\perp}M$,

it is clear from above that CR-submanifolds and slant submanifolds are hemi-slant submanifolds with slant angle $\theta = \frac{\pi}{2}$ and $D^{\theta}$ = {0}, respectively.
\end{definition}

In the rest of this paper, we use $M$ a hemi-slant submanifold of almost contact metric manifold $N$.

On the other hand, if we denote the dimensions of the distributions $D^{\perp}$ and $D^{\theta}$ by $m_{1}$ and $m_{2}$ respectively, then we have the following cases:

(1) If  $m_{2} = 0$, then $M$ is anti-invariant  submanifold,

(2) If $m_{1}=0$ and $\theta = 0$, then $M$ is an invariant submanifold,

(3) If $m_{1}=0$ and $\theta \ne 0$, then $M$ is a proper slant submanifold with slant angle $\theta$,

(4) if $m_{1}, m_{2} \ne 0$ and $\theta \in (0, \frac{\pi}{2})$, then $M$ is a proper hemi-slant submanifold.
\newline

Suppose $M$ to be a hemi-slant submanifold of an almost contact metric manifold $N$, then for any $X \in TM$, we put
\begin{eqnarray}\label{h1}
X = P_{1}X + P_{2}X + \eta(X)\xi,
\end{eqnarray}
where $P_{1}$ and $P_{2}$ are projection maps on the distribution $D^{\perp}$ and $D^{\theta}$. Now operating $\phi$ on both sides of (\ref{h1}), we arrive at
\begin{eqnarray*}
\phi X = \phi P_{1}X + \phi P_{2}X + \eta(X)\phi\xi,
\end{eqnarray*}
Operating $\phi$ on both sides, we get
\begin{eqnarray*}
 TX + FX = FP_{1}X + T P_{2}X + F P_{2}X,
\end{eqnarray*}
It is easy to see on comparing that
\begin{eqnarray*}
 TX = T P_{2}X, \hspace{1cm}  FX = FP_{1}X + F P_{2}X,
\end{eqnarray*}
If we denote the orthogonal complement of $\phi TM$ in $T^{\perp}M$ by $\mu$, then the normal bundle $T^{\perp} M$ can be decomposed as
\begin{eqnarray}\label{h2}
T^{\perp}M = F(D^{\perp})\oplus F(D^{\theta})\oplus \mu.
\end{eqnarray}
As $N(D^{\perp})$ and $N(D^{\theta})$ are orthogonal distributions on $F$ . g(Z,W) = 0 for each $Z \in D^{\perp}$ and $W \in D^{\theta}$. Thus, by (\ref{a1}),(\ref{a3}) and (\ref{a9}), we have
\begin{eqnarray}
g(FZ, FX) = g(\phi Z, \phi X) = g(Z,X) = 0,
\end{eqnarray}
which shows that the distributions $F(D^{\perp})$ and $F(D^{\theta})$ are mutually perpendicular. In fact, the decomposition (\ref{h2}) is an orthogonal direct decomposition.
\begin{lemma}
Let $M$ be a hemi-slant submanifolds of a cosymplectic manifold $N$. Then we have
\begin{eqnarray*}
\nabla_{X}TY - A_{FY}X = T\nabla_{X}Y + th(X,Y)
\end{eqnarray*}
and
\begin{eqnarray*}
h(X,TY) + \nabla_{X}^{\perp}FY = F\nabla_{X}Y + fh(X,Y)
\end{eqnarray*}
for all $X, Y$ $\in$ $TM$.
\end{lemma}
\begin{lemma}
Let $M$ be a hemi-slant submanifolds of a cosymplectic manifold $N$. Then we have
\begin{eqnarray*}
\nabla_{X}tV - A_{FV}X = - T A_{V}X + t\nabla_{X}^{\perp}V
\end{eqnarray*}
and
\begin{eqnarray*}
h(X,tV) + \nabla_{X}^{\perp}FV = - f A_{V}Y + f\nabla_{V}^{\perp}V.
\end{eqnarray*}
for all $X$ $\in$ $TM$ and $V \in T^{\perp}M$.
\end{lemma}
\begin{lemma}
Let $M$ be a hemi-slant submanifolds of a cosymplectic manifold $N$, then
\begin{eqnarray*}
h(X,\xi) = 0, \hspace{1cm} h(TX,\xi) = 0 \hspace{1cm} \nabla_{X}\xi = 0,
\end{eqnarray*}
for all $X, Y$ $\in$ $TM$.
\end{lemma}
\begin{proof}
We know that for $\xi \in TM$, we have
\begin{eqnarray*}
\overline{\nabla}_{X}\xi = \nabla_{X}\xi + h(X,\xi)
\end{eqnarray*}
From (\ref{a4}), it follows that
\begin{eqnarray*}
 \nabla_{X}\xi + h(X,\xi) = 0.
\end{eqnarray*}
Thus result follows directly from the above equation.
\end{proof}
\begin{theorem}
Let $M$ be a hemi-slant submanifold of a cosymplectic manifold $N$, Then
\begin{eqnarray*}
A_{\phi X}Y = A_{\phi Y}X,
\end{eqnarray*}
for all $X, Y \in D^{\theta}$.
\end{theorem}
\begin{proof}
Using (\ref{a7}), we have
\begin{eqnarray*}
g(A_{\phi X}Y, Z) &=& g(h(Y,Z),\phi X)\\
&=& -g(\phi h(X,Z), X)\\
&=& -g(\phi \overline{\nabla}_{Z}Y, X) - g(\phi \nabla_{Z}Y, X)\\
&=& -g(\phi \overline{\nabla}_{Z}Y, X).
\end{eqnarray*}
Whereby using (\ref{a4}), we have
\begin{eqnarray*}
g(A_{\phi X}Y, Z) &=& -g( - A_{\phi Y}Z + \nabla^{\perp}_{Z}\phi Y, X).
\end{eqnarray*}
By use of  $h(X,Y) = h(Y, X)$, we arrive at
\begin{eqnarray*}
g(A_{\phi X}Y, Z) &=& g(A_{\phi Y}X, Z)
\end{eqnarray*}
Hence the result.
\end{proof}
\begin{theorem}
Let $M$ be a submanifold of a cosymplectic manifold $N$. Then the distribution $D^{\perp}$ is integrable if and only if
\begin{eqnarray}\label{t1}
A_{FZ}W = A_{FW}Z,
\end{eqnarray}
for any $Z, W$ in $D^{\perp}$.
\end{theorem}
\begin{proof}
For $Z$,$W$ $\in$ $D^{\perp}$, by using (\ref{a4}), we have
\begin{eqnarray*}
(\overline{\nabla}_{Z}\phi)W = 0,
\end{eqnarray*}
which implies that
\begin{eqnarray*}
\overline{\nabla}_{Z} \phi W - \phi \overline{\nabla}_{Z} W = 0.
\end{eqnarray*}
Using (\ref{a5}), (\ref{a6}), (\ref{a7}) and (\ref{a8}), we have
\begin{eqnarray*}
\overline{\nabla}_{Z}FW - T\overline{\nabla}_{Z}W - F\overline{\nabla}_{Z}W = 0,
\end{eqnarray*}
or
\begin{eqnarray}\label{t2}
-A_{FW}Z - \nabla_{Z}^{\perp}FW - T\nabla_{Z}W + th(Z,W)- F{\nabla}_{Z}W - nh(Z,W) = 0,
\end{eqnarray}
Comparing the tangential components of (\ref{t2}), we have
\begin{eqnarray*}
A_{FW}Z + T\nabla_{Z}W + th(Z,W)  = 0,
\end{eqnarray*}
Interchange $Z$ and $W$, and subtract, we have
\begin{eqnarray*}
T[Z,W] = A_{FW}Z - A_{FZ}W.
\end{eqnarray*}
Thus $[Z,W] \in D^{\perp}$ if and only if  (\ref{t1}) is satisfied
\end{proof}
\begin{theorem}
Let $M$ be a hemi-slant submanifold of a cosymplectic manifold $N$. Then the distribution $D^{\theta}\oplus D^{\perp}$ is integrable iff
\begin{eqnarray*}
g([X,Y],\xi) = 0,
\end{eqnarray*}
for all $X, Y \in D^{\theta}\oplus D^{\perp}$
\end{theorem}
\begin{proof}
For $X, Y \in D^{\theta}\oplus D^{\perp}$, we have
\begin{eqnarray*}
g([X,Y],\xi)] &=& g(\nabla_{X}Y,\xi) - g(\nabla_{Y}X, \xi)\\
&=& -g(\nabla_{X}\xi, Y) + g(\nabla_{Y}\xi,X)\\
\end{eqnarray*}
Using (\ref{a4}), we have
$$g([X,Y],\xi) = 0$$.
\end{proof}
\begin{theorem}
Let $M$ be a hemi-slant submanifold of a cosymplectic manifold $N$. Then the anti-invariant  distribution $D^{\perp}$ is integrable if and only if
\begin{eqnarray}\label{t5}
T\nabla_{Z}W = T\nabla_{W}Z,
\end{eqnarray}
for any $X, Y \in  D^{\perp}$.
\end{theorem}
\begin{proof}
For $X,Y \in D^{\perp}$, we have
\begin{eqnarray*}
(\nabla_{Z}\phi)W = 0,
\end{eqnarray*}
or
\begin{eqnarray*}
\overline{\nabla}_{Z} \phi W - \phi \overline{\nabla}_{Z} W = 0,
\end{eqnarray*}
whereby we have
\begin{eqnarray*}
\overline{\nabla}_{Z} FW - \phi (\nabla_{Z} W + h(W,Z)) = 0,
\end{eqnarray*}
or
\begin{eqnarray*}
-A_{FW}Z + \overline{\nabla}_{Z}^{\perp}FW - T\nabla_{Z}W - F{\nabla}_{Z}W - th(Z,W) - nh(Z,W) = 0.
\end{eqnarray*}
Comparing the tangential components we have,
\begin{eqnarray*}
-A_{FW}Z  - T\nabla_{Z}W - th(Z,W) = 0,
\end{eqnarray*}
from which we conclude that
\begin{eqnarray*}
T[Z,W] = A_{FW}Z  + T\nabla_{Z}W + th(Z,W).
\end{eqnarray*}
For $[Z,W] \in D^{\perp}$, we have $\phi[Z,W] = F[Z,W]$ because the tangential component of $\phi[Z,W]$ is zero. Thus, we have
\begin{eqnarray}\label{t3}
 A_{FW}Z  + T\nabla_{Z}W + th(Z,W) = 0.
\end{eqnarray}
Similarly, we have
\begin{eqnarray}\label{t4}
A_{FZ}W  + T\nabla_{W}Z + th(W,Z) = 0.
\end{eqnarray}
Whereby use of Theorem 3, (\ref{t3}), and (\ref{t4}), we have
\begin{eqnarray*}
T\nabla_{Z}W = T\nabla_{W}Z
\end{eqnarray*}
Thus the anti-invariant distribution $D^{\perp}$ is integrable if and only if (\ref{t5}) is satisfied.
\end{proof}
\begin{theorem}
Let $M$ be a hemi-slant submanifold of a cosymplectic manifold $N$. Then the slant distribution $D^{\theta}$ is integrable iff
\begin{eqnarray*}
h(X,TY) - h(Y,TX) + \nabla_{X}^{\perp}FY - \nabla_{Y}^{\perp} FX \in \mu  \oplus F(D^{\theta}),
\end{eqnarray*}
for any $X,Y \in  D^{\theta}$.
\end{theorem}
\begin{proof}
For $Z \in D^{\perp}$  and $X, Y \in  D^{\theta}$,  we have
\begin{eqnarray*}
g([X,Y], Z) = g( \overline{\nabla}_{X}Y - \overline{\nabla}_{Y}X, Z).
\end{eqnarray*}
Using (\ref{a1}), (\ref{a2}) and (\ref{a4}), we get
\begin{eqnarray*}
g([X,Y], Z) = g(\phi \overline{\nabla}_{X}Y, \phi Z) - g(\phi \overline{\nabla}_{Y}X, \phi Z)
\end{eqnarray*}
whereby use of (\ref{a5}), (\ref{a6}), we obtain
\begin{eqnarray*}
g([X,Y], Z) = g(h(X,TY) - h(Y,TX) + \nabla_{X}^{\perp}FY - \nabla_{Y}^{\perp} FX, \phi Z)
\end{eqnarray*}
As $\phi X \in \phi(D^{\perp})$ and $F(D^\theta)$ and $F(D^{\perp}$ are orthogonal to each other in $T^{\perp}M$, thus we conclude the result.
\end{proof}
\begin{theorem}
Let $M$ be a hemi-slant submanifold of a cosymplectic manifold $N$. Then the slant distribution $D^{\theta}$ is integrable if and only if
\begin{eqnarray*}
P_{1}\{\nabla_{X}TY - \nabla_{Y}TX - A_{FX}Y - A_{FY}X\} = 0,
\end{eqnarray*}
for any $X, Y \in  D^{\theta}$.
\end{theorem}

\begin{proof}
We denote by $P_{1}$ and $P_{2}$ the projections on $D^{\perp}$ and $D^{\theta}$ respectively. For any vector fields $X$, $Y$ $\in$ $D^{\theta}$. Using equation (\ref{a4}), we have
\begin{eqnarray*}
(\overline{\nabla}_{X}\phi)Y = 0,
\end{eqnarray*}
that is
\begin{eqnarray*}
(\overline{\nabla}_{X}\phi Y) - \phi\overline{\nabla}_{X}Y = 0.
\end{eqnarray*}
By using equation (\ref{a5}),(\ref{a6}) and (\ref{a9}), we have
\begin{eqnarray*}
\overline{\nabla}_{X} TY + (\overline{\nabla}_{X} FY) - \phi(\nabla_{X} Y + h(X,Y)),
\end{eqnarray*}
or
\begin{eqnarray*}\label{t6}
\nabla_{X} TY + h(X,TY) - A_{FY}X + \nabla_{X}^{\perp}FY - T\nabla_{X}Y - F\nabla_{X}Y - th(X,Y) - nh(X,Y) = 0.
\end{eqnarray*}
Comparing the tangential components of (\ref{t6}), we have
\begin{eqnarray}\label{t7}
\nabla_{X} TY - A_{FY}X - T\nabla_{X}Y - th(X,Y) = 0.
\end{eqnarray}
Replacing $X$ and $Y$, we infer
\begin{eqnarray}\label{t8}
\nabla_{Y} TX - A_{FX}Y - T\nabla_{Y}X - th(Y,X) = 0.
\end{eqnarray}
From (\ref{t7}) and (\ref{t8}), we arrive at
\begin{eqnarray}\label{t9}
T[X,Y] = \nabla_{X} TY - \nabla_{Y} TX + A_{FY}X - A_{FX}Y.
\end{eqnarray}
Applying  $P_{1}$ to (\ref{t9}), we obtain the result.
\end{proof}
\begin{theorem}
Let $M$ be a  hemi-slant submanifold of a cosymplectic manifold $N$. If the leaves of $D^{\perp}$ are totally geodesic in $M$, then
\begin{eqnarray*}
A_{FW}Z = \nabla_{Z}TW,
\end{eqnarray*}
for $X \in D^{\theta}$ and $Z,W \in D^{\perp}.$
\end{theorem}
\begin{proof}
Since $(\overline{\nabla}_{Z}\phi)W = 0$. From(\ref{a4}), we have
\begin{eqnarray*}
\overline{\nabla}_{Z}\phi W = \phi\overline{\nabla}_{Z}W
\end{eqnarray*}
Using (\ref{a5}), (\ref{a6}) and (\ref{a9}), we obtain
\begin{eqnarray*}
\nabla_{Z}TW + h(Z,TW) - A_{FW}Z + \nabla_{Z}^{\perp}FW  = \phi\nabla_{Z}W + \phi h(Z,W).
\end{eqnarray*}
For $X \in D^{\theta}$, we have
\begin{eqnarray*}
g(\nabla_{Z}TW, X) - g(A_{FW}Z, X) = g(\phi\nabla_{Z}W, X).
\end{eqnarray*}
Therefore, we have
\begin{eqnarray}\label{le1}
g(\nabla_{Z}W, \phi X) = g(A_{FW}Z - \nabla_{Z}TW, X).
\end{eqnarray}
The leaves of $D^{\perp}$ are totally geodesic in $M$, if for $Z, W \in D^{\perp}$, $\nabla_{Z}W \in D^{\perp}$. Therefore from (\ref{le1}), we get the result.
\end{proof}
\begin{theorem}
Let $M$ be a hemi-slant submanifold of a cosymplectic manifold $N$. If the leaves of $D^{\theta}$ are totally geodesic in $M$, then
\begin{eqnarray*}
\nabla_{X}\phi Y = \phi h(X,Y),
\end{eqnarray*}
for $X, Y \in D^{\theta}$ and $Z \in D^{\perp}.$
\end{theorem}
\begin{proof}
From  $(\ref{a4})$, we know that $(\overline{\nabla}_{X}\phi)Y = 0$, then
\begin{eqnarray*}
\overline{\nabla}_{X}\phi Y = \phi\overline{\nabla}_{X}Y.
\end{eqnarray*}
For $Z \in D^{\perp}$  and using (\ref{a5}), (\ref{a6}) and (\ref{a9}), we get
\begin{eqnarray*}
g(\nabla_{X}\phi Y, Z) - g(\phi\nabla_{X}Y,Z) = g(h(X,Y),\phi Z).
\end{eqnarray*}
Therefore from above equation, we get the result.
\end{proof}
\begin{theorem}
Let $M$ be a totally umbilical hemi-slant submanifold of a cosymplectic manifold $N$. Then at least one of the following holds
\newline
(1)$dim(D^{\perp})$ = 1,
\newline
(2)H $\in \mu $,
\newline
(3)M is proper hemi-slant submanifold.
\end{theorem}
\begin{proof}
For a cosymplectic manifold, we have
\begin{eqnarray*}
(\overline \nabla_{Z}\phi)Z = 0,
\end{eqnarray*}
for any $Z \in D^{\perp}$.
Using (\ref{a5}),(\ref{a6}) and (\ref{a9}), we have
\begin{eqnarray*}
\overline \nabla_{Z}FZ - \phi(\nabla_{Z}Z + h(Z,Z)) = 0.
\end{eqnarray*}
Whereby, we obtain
\begin{eqnarray*}
-A_{FZ}Z + \nabla^{\perp}_{Z}FZ - F\nabla_{Z}Z - th(Z,Z) - nh(Z,Z)= 0.
\end{eqnarray*}
Comparing the tangential components, we have
\begin{eqnarray*}
A_{FZ}Z + th(Z,Z)= 0.
\end{eqnarray*}
Taking inner product with $W \in D^{\perp}$, we obtain
\begin{eqnarray*}
g(A_{FZ}Z + th(Z,Z),W)= 0,
\end{eqnarray*}
or
\begin{eqnarray*}
g(h(Z,W),FZ) + g(th(Z,Z),W) = 0.
\end{eqnarray*}
Since $M$ is totally umbilical submanifold, we obtain
\begin{eqnarray*}
g(Z,W)g(H,FZ) + g(Z,Z)g(tH,W) = 0.
\end{eqnarray*}
The above equation has a solutio if either $dim(D^{\perp})$ = 1 or H $\in \mu $ or $D^{\perp}$ = 0, this completes the proof.
\end{proof}

{\bf{References}}

\end{document}